\documentclass[graybox]{svmult}

\usepackage{mathptmx}       

\usepackage{helvet}         

\usepackage{courier}        

\usepackage{type1cm}        

\usepackage{makeidx}         

\usepackage{graphicx}        

\usepackage{multicol}        

\usepackage{amssymb,amsfonts,amscd,amsmath}

\input amssym.def

\input amssym.tex

\makeindex             

\newtheorem{thm}{Theorem}[section]

\newtheorem{lem}[thm]{Lemma}

\newcommand\bi{\begin{itemize}}

\newcommand\ei{\end{itemize}}

\newcommand\ben{\begin{enumerate}}

\newcommand\een{\end{enumerate}}

\newcommand\be{\begin{equation}}

\newcommand\ee{\end{equation}}

\newcommand\bea{\begin{eqnarray}}

\newcommand\eea{\end{eqnarray}}

\numberwithin{subsubsection}{subsection}

\begin{document}

\title*{Generalized Ramanujan Primes}



\author{Nadine Amersi, Olivia Beckwith, Steven J. Miller, Ryan Ronan and Jonathan Sondow}

\institute{Department of Mathematics, University College London, n.amersi@ucl.ac.uk; Department of
Mathematics, Harvey Mudd College, obeckwith@gmail.com; Department of Mathematics and Statistics,
Williams College, sjm1@williams.edu (Steven.Miller.MC.96@aya.yale.edu); Department of Electrical
Engineering, Cooper Union, ronan2@cooper.edu; and 209 West 97 Street, New York, NY 10025,
jsondow@alumni.princeton.edu}

%




%

\maketitle

\abstract{In 1845, Bertrand conjectured that for all integers $x\ge2$, there exists at least one prime in $(x/2, x]$. This was proved by Chebyshev in 1860, and then generalized by Ramanujan in 1919. He showed that for any $n\ge1$, there is a (smallest) prime $R_n$ such that $\pi(x)- \pi(x/2) \ge n$ for all $x \ge R_n$. In 2009 Sondow called $R_n$ the $n$th Ramanujan prime and proved the asymptotic behavior $R_n \sim p_{2n}$ (where $p_m$ is the $m$th prime). In the present paper, we generalize the interval of interest by introducing a parameter $c \in (0,1)$ and defining the $n$th $c$-Ramanujan prime as the smallest integer $R_{c,n}$ such that for all $x\ge R_{c,n}$, there are at least $n$ primes in $(cx,x]$. Using consequences of strengthened versions of the Prime Number Theorem, we prove that $R_{c,n}$ exists for all $n$ and all $c$, that $R_{c,n} \sim p_{\frac{n}{1-c}}$ as $n\to\infty$, and that the fraction of primes which are $c$-Ramanujan converges to $1-c$. We then study finer questions related to their distribution among the primes, and see that the $c$-Ramanujan primes display striking behavior, deviating significantly from a probabilistic model based on biased coin flipping; this was first observed by Sondow, Nicholson, and Noe in the case $c = 1/2$. This model is related to the Cramer model, which correctly predicts many properties of primes on large scales, but has been shown to fail in some instances on smaller scales.
 }





\section{Introduction}


For $n\ge1$, the $n$th \emph{Ramanujan prime} was defined by Sondow \cite{So} as the smallest positive integer $R_n$ with the property that for any $x \ge R_n$, there are at least $n$ primes in the interval $(\frac12x, x]$. By its minimality, $R_n$ is indeed a prime, and the interval $\left(\frac12R_n,R_n\right]$ contains exactly $n$ primes.

In $1919$ Ramanujan \cite{Ra} proved a result which implies that $R_n$ exists, and he gave the first five Ramanujan primes as $R_n=2,11,17,29, 41$ for $n=1,2,3,4,5$, respectively (see sequence A104272 at \cite{OEIS}). The case $R_1=2$ is \emph{Bertrand's Postulate} (proved by Chebyshev): for all $x \ge 2$, there exists a prime $p$ with $\frac12x<p\le x$. Sondow proved that $R_n \sim p_{2n}$ as $n\to\infty$ (where $p_m$ is the $m$th prime), and he and Laishram \cite{La} proved the bounds $p_{2n}<R_n<p_{3n},$ respectively, for $n>1.$

In the present article, we generalize the notion of Ramanujan primes. Instead of studying the intervals $(\frac12x, x]$, we consider the intervals $(cx, x]$ for a fixed number $c \in (0,1)$ (these were also investigated by Shevelev \cite{Sh}, whose parameter $k$ equals $1/c$). Namely, the $n$th \emph{$c$-Ramanujan prime} is defined to be the smallest positive integer $R_{c,n}$ such that for any $x \ge R_{c,n}$ there are at least $n$ primes in the interval $(cx, x]$. Here, too, the minimality implies that $R_{c,n}$ is a prime and $\pi(R_{c,n}) - \pi(c R_{c,n})\ =\ n$ (where $\pi(x)$ is number of primes at most $x$). Note that  $R_{c_1,n} \leq R_{c_2,n}$ for $c_1 < c_2$. When $c=1/2$, we recover $R_{1/2,n}=R_n$, the $n$th Ramanujan prime. Thus  $R_{c,n} \leq R_n$ if $c < 1/2$.

We also determine the $c$-dependence of the generalizations of certain results in \cite{So,La,SNN}.

We quickly review notation. We denote the number of $c$-Ramanujan primes at most $x$ by $\pi_c(x)$, and let $p_m$ denote the $\lfloor m \rfloor$th prime. We write ${\rm Li}(x)$ for the logarithmic integral, given by \be  {\rm Li}(x) \ = \ \int_2^x \frac{dt}{\log t}. \ee By $f(x) \ll g(x)$, which we often write as $f(x) = O(g(x))$, we mean there exist constants $x_0$ and $C > 0$ such that for all $x \ge x_0$ we have $|f(x)| \le Cg(x)$, while by $f(x) = o(g(x))$ we mean that $\lim_{x\to \infty} f(x)/g(x) = 0$.

The existence of $R_{c,n}$ follows from the Prime Number Theorem; we give a proof in Theorem \ref{thm:rcnexists} of \S\ref{sec:asympbehavior}. Our main result is the $c$-dependence of $R_{c,n}$.

\begin{thm}[Asymptotic behavior of $R_{c,n}$]\label{thm:asympbehRcn} We have:

\begin{enumerate}

\label{density} \item For any fixed $c\in (0,1)$, the $n${\rm th} $c$-Ramanujan prime is asymptotic to the $\frac{n}{1-c}${\rm th} prime as $n\to\infty$, that is,\be \lim_{n\to\infty} \frac{R_{c,n}}{p_{\frac{n}{1-c}}} = 1.\ee More precisely, there exists a constant $\beta_{1,c}> 0$ such that \be |R_{c,n} - p_\frac{n}{1-c}| \le \beta_{1,c}\, n \log\log n\ee for all sufficiently large $n$.

\label{asymp} \item In the limit, the probability of a generic prime being a $c$-Ramanujan prime is $1-c$. More precisely, there exists a constant $\beta_{5,c}$ such that for $N$ large we have \be\left|\frac{\pi_c(N)}{\pi(N)} - (1-c)\right| \ \le \ \frac{\beta_{5,c} \log\log N}{\log N}.\ee

\end{enumerate}

\end{thm}

The proof uses the Prime Number Theorem, and is given in \S\ref{sec:asympbehavior}.

For example the first thirty-six $\frac{1}{4}$-Ramanujan primes are 2, 3, 5, 13, 17, 29, 31, 37, 41, 53, 59, 61, 71, 79, 83, 97, 101, 103, 107, 127, 131, 137, 149, 151, 157, 173, 179, 191, 193, 197, 199, 223, 227, 229, 239, 251 (sequence A193761 at \cite{OEIS}), and the first thirty-six $\frac{3}{4}$-Ramanujan primes are 11, 29, 59, 67, 101, 149, 157, 163, 191, 227, 269, 271, 307, 379, 383, 419, 431, 433, 443, 457, 563, 593, 601, 641, 643, 673, 701, 709, 733, 827, 829, 907, 937, 947, 971, 1019 (sequence A193880 at \cite{OEIS}).

We end with some numerical results about the distribution of $c$-Ramanujan primes in the sequence of primes, extending calculations from \cite{So} and Sondow, Nicholson and Noe \cite{SNN} in the case $c=1/2.$ For small values of $c$, the length of the longest run of $c$-Ramanujan primes among the primes in $(10^5,10^6)$ is less than expected (e.g., for $c=0.05,$ we observe a longest run of length 97, but we expect 127). For values of $c$ near $1,$ the opposite behavior is observed: the length of the longest run is greater than expected (e.g., for $c=0.90$ we expect the longest run of consecutive non-$c$-Ramanujan primes to have length 91, but the actual length is 345). The expected lengths were computed using a coin flip model with fixed probability $P_c(n)$ of a prime in the interval $[10^n, 10^{n+1})$ being $c$-Ramanujan; see \cite{Sc} for a full description of the theory and results of such a model.

\ \\

The authors thank the participants of the 2011 CANT conference for many useful conversations. The first, second and fourth named authors were partially supported by NSF grant DMS0850577 and Williams College (the first named author was additionally supported by the Mathematics Department of University College London); the third named author was partially supported by NSF grant DMS0970067.




\section{Asymptotic Behavior of Generalized Ramanujan Primes}\label{sec:asympbehavior}

To simplify the exposition we use the Prime Number Theorem below, though weaker bounds (such as Rosser's Theorem) would suffice for many of the results.

\begin{thm}[Prime Number Theorem]\label{thm:cmprimenumbertheorem}
There is a positive constant $\gamma_1 < 1/2$ such that \be \pi(x) \ = \ \text{{\rm Li}}(x)  \ + \ O\left(x \cdot\exp\left(-\gamma_1\sqrt{\log x}\right)\right) \quad (x\to\infty).\label{eq: pnt} \ee  In particular, for some numbers $\gamma_2 > 0$ and $x_0 > 0,$ we have \be \left|\pi(x) - {\rm Li}(x)\right| \ \le \ \gamma_2 \frac{x}{\log^5 x} \quad (x\ge x_0).\label{eq: pnt ineq}\ee
\end{thm}

\begin{proof} See \cite{IK} for a proof of \eqref{eq: pnt}. Taylor expanding the exponential factor in \eqref{eq: pnt}, we see that it decays faster than any power of the logarithm, and thus \eqref{eq: pnt ineq} follows.\hfill $\Box$\end{proof}

We will also have occasion to use the following strengthened version of Rosser's theorem (see for example page 233 of \cite{BS}): \be\label{eq:pmvsmlogmfirst} \left|p_m - (m \log m + m\log\log m)\right| \ \le \ m \ee for $m \ge 6$; however, for our purposes the following weaker statement often suffices: \be\label{eq:pmvsmlogm} p_m \ = \  m\log m + O(m\log\log m).\ee

The following result shows that $c$-Ramanujan primes exist. Later we'll determine their asymptotic behavior and study their distribution in the sequence of all primes.

\begin{thm}[Existence of $R_{c,n}$]\label{thm:rcnexists}
For any $c \in (0,1)$ and any positive integer $n$, the $c$-Ramunjan prime $R_{c,n}$ exists.
\end{thm}

\begin{proof}  By Theorem \ref{thm:cmprimenumbertheorem} and the Mean Value Theorem, if $x$ is sufficiently large, then for some point $y_c = y_c(x)\in [cx,x]$ we have\bea\label{eq:pipicx} \pi(x) - \pi(cx) & \ = \ & {\rm Li}(x) - {\rm Li}(cx) + O(x \log^{-5} x) \nonumber\\ &=& {\rm Li'}(y_c) (x - cx) + O(x \log^{-5} x) \nonumber\\ &=& \frac{(1-c)x}{\log y_c} + O(x \log^{-5} x). \eea Since $\log y_c = \log x - b_c$, where $b_c = b_c(x) \in [0,-\log c]$, we get \be\label{eq:piminuspicx} \pi(x) - \pi(cx) \ = \ \frac{(1-c)x}{\log x - b_c} + O\left(\frac{x}{\log^5 x}\right) \ = \ \frac{(1-c)x}{\log x} + O\left(\frac{x}{\log^2 x}\right), \ee which is asymptotic to $(1-c)x/\log x$ as $x\to\infty$. Hence $\pi(x) - \pi(cx) \ge n,$ for all $x$ sufficiently large, and the theorem follows. \hfill $\Box$ \end{proof}















%



Before proving Theorem \ref{thm:asympbehRcn}, we derive some crude but useful bounds on $\log R_{c,n}$. While we could derive stronger bounds with a little more work, the present ones give sufficient estimates for our later analysis of $R_{c,n}$.

\begin{lem}\label{lem:boundslogRcn} For any $c \in (0,1)$, there exist constants $\beta_{2,c}>0$ and $N_c > 0$ such that, for all $n \ge N_c,$
\begin{equation}\label{eq:lemmaboundslogrcn}
\left( 1- \frac{\beta_{2,c}\log\log n}{\log n}\right)\log n\ \le \ \log R_{c,n}\ \le \ \left(1+\frac{\beta_{2,c}\log\log n}{\log n}\right)\log n.
\end{equation}
\end{lem}

\begin{proof}
We first show that the following inequality holds for sufficiently large $n$:
\begin{equation} \label{eq: c-R bounds}
n\log n\ \le \ R_{c,n}\ \le \ \frac{2n}{1-c}\log \frac{2n}{1-c}.
\end{equation}
The lower bound follows from the trivial observation that $p_n \leq R_{c,n}$ for all $c$ and all $n$, and Rosser's Theorem \cite{Ro}, which states that $n\log n < p_n$.

To obtain the upper bound, we show that there exists a constant $\alpha_c > 0$ such that for large $n$ we have $R_{c,n} \leq \alpha_c n \log  (\alpha_c n)$. (It is trivial to find such a constant if we allow $\alpha_c$ to depend on $n$ and $c$, but for our applications we need a bound independent of $n$, though it may depend on $c$.)








From \eqref{eq:piminuspicx}, we see that, for some $N_{1,c}$ (which may depend on $c$ but is independent of $n$), if $x \geq N_{1,c}$, then \be \pi(x) - \pi(cx)\ >\ \frac{2(1-c)}{3} \frac{x}{\log x}.\ee

We now show that \be \alpha_c\ =\ c_1\ :=\ \frac{2}{1-c}\ee suffices to have  $R_{c,n} \leq \alpha_c n \log  (\alpha_c n)$. To see this, take $x >c_1 n \log  (c_1 n)$. Then as $ \frac{x}{\log  x} $ is increasing when $x>e$, we have
\bea
\pi(x) - \pi(cx) & \ > \ & \frac{2(1-c)}{3} \frac{x}{\log  x}\nonumber\\ & >& \frac{2(1-c)}{3} \frac{ c_1n \log ( c_1n)}{\log  (c_1n \log ( c_1n))} \nonumber\\ &=& \frac{4n}{3\left(1 + \dfrac{\log  (\log  (c_1n))}{\log  (c_1n)}\right)}.
\eea
As $\lim_{y \to \infty} \frac{\log  \log  y}{\log  y} = 0 $, there is an $N_{2,c}$ such that for all $n \geq N_{2,c}$ we have \be\frac{4}{3}\ >\ 1 + \frac{\log  (\log  (c_1n))}{\log  (c_1n)}.\ee Taking $N_{3,c} = \max(N_{1,c}, N_{2,c}),$ we see that for $n \geq N_{3,c}$ we have
\begin{equation}
x\ >\  c_1n \log (c_1n)\ \ \ \implies\ \ \ \pi(x) - \pi(cx)\ \geq\ n.
\end{equation}
Thus for $n$ sufficiently large ($n > N_{3,c}$) we find that \be R_{c,n}\ \le \ c_1n \log \left(c_1n\right), \ee which completes the proof of \eqref{eq: c-R bounds}.

Taking logarithms in \eqref{eq: c-R bounds} yields
\begin{equation} \label{eq: log bounds}
\log (n\log n)\ \le\ \log R_{c,n}\ \le\ \log \left(\frac{2n}{1-c}\log \frac{2n}{1-c} \right).
\end{equation}
The rightmost term is \bea
\log \left(\frac{2n}{1-c}\log \frac{2n}{1-c} \right)& \ = \ & \left(1 + \frac{\log\log n + \log \frac{2}{1-c} + \log\log \frac{2}{1-c}}{\log  n}\right)\log n \nonumber\\ & \le & \left(1 + \frac{\beta_{2,c} \log\log n}{\log n}\right) \log n,\eea for some $\beta_{2,c} > 0$ and all $n$ sufficiently large, say $n \ge N_{4,c}$. The leftmost term in \eqref{eq: log bounds} is \bea \log(n \log n) \ = \ & \left(1 + \dfrac{\log\log n}{\log n}\right)\log n \ >\ \left( 1- \dfrac{\beta_{2,c}\log\log n}{\log n}\right)\log n. \eea Taking $N_c := \max(N_{3,c}, N_{4,c})$, the proof of the lemma is complete. \hfill $\Box$\\
\end{proof}



We now turn to the proof of Theorem \ref{thm:asympbehRcn}. We first prove the claimed asymptotic behavior (part 1 of the theorem), and then prove the limiting percentage of primes that are $c$-Ramanujan is $\frac1{1-c}$ (part 2 of the theorem).\\

\noindent \emph{Proof of Theorem \ref{thm:asympbehRcn}, part 1.} Since $ \pi(R_{c,n}) - \pi(c R_{c,n})\ =\ n$, taking $x = R_{c,n}$ in \eqref{eq:piminuspicx} and multiplying by $(1-c)^{-1}\log R_{c,n}$ yields \be\label{eq:fracn1minuscsubtract} \frac{n}{1-c}\log R_{c,n} \ = \ R_{c,n} + O\left(\frac{R_{c,n}}{\log R_{c,n}}\right). \ee

Equivalently, there is a constant $\gamma_{3,c}$ such that \be\label{eq:firstrcnminusndot} \left|\frac{n}{1-c} \log R_{c,n} - R_{c,n}\right| \ \le \ \gamma_{3,c} \frac{R_{c,n}}{\log R_{c,n}}. \ee

On the other hand, using the bounds on $\log R_{c,n}$ from \eqref{eq:lemmaboundslogrcn}, we find that\be \left|\frac{n}{1-c}\log R_{c,n} - \frac{n}{1-c} \log n\right| \ \le \ \frac{n}{1-c} \beta_{2,c}\log\log n. \ee For $m \ge 20$, from \eqref{eq:pmvsmlogmfirst} we have \be \left|p_m - m\log m\right| \ \le \ 2 m\log\log m; \ee we use this with $m = \frac{n}{1-c}$ and note \be \left|\frac{n}{1-c}\log\frac{n}{1-c} - \frac{n}{1-c}\log n\right| \ = \  O_c(n).\ee We now bound the distance from $R_{c,n}$ to $p_{\frac{n}{1-c}}$ by the triangle inequality and the above bounds: \bea\label{eq:closenessrcnpn} \left|R_{c,n} - p_{\frac{n}{1-c}}\right| & \ \le \ & \left|R_{c,n} - \frac{n}{1-c}\log R_{c,n}\right| + \left|\frac{n}{1-c}\log R_{c,n} - \frac{n}{1-c} \log n\right|\nonumber\\ & & \ \ \ +\ \left|\frac{n}{1-c}\log n - \frac{n}{1-c}\log \frac{n}{1-c}\right| + \left|\frac{n}{1-c}\log n - p_{\frac{n}{1-c}}\right| \nonumber\\ & \le & \beta_{3,c} n\log\log n \eea (as each of the four terms is $O(n\log\log n)$, with the first term's bound following from using $R_{c,n} \ll n\log n$ in \eqref{eq:fracn1minuscsubtract}). As $(n\log\log n) / p_n \to 0$, we see $R_{c,n}$ is asymptotic to $p_{\frac{n}{1-c}}$. \hfill $\Box$

\ \\



\noindent \emph{Proof of Theorem \ref{thm:asympbehRcn}, part 2.} Heuristically, if $R_{c,n}$ were exactly the $\frac{n}{1-c}$th prime, this would mean that one out of every $\frac{1}{1-c}$ primes is $c$-Ramanujan, and thus the density of $c$-Ramanujan primes amongst the prime numbers would be $1-c$.  We now make this heuristic precise.

Let $N$ be an integer, and choose $n$ so that $\lfloor\frac{n}{1-c}\rfloor = N$, so $n$ is essentially $(1-c)N$. For each $N$ we need to show that the number of $c$-Ramanujan primes at most $N$ is $\left((1-c) + o_c(1)\right)\pi(N)$, where $o_c(1)\to 0$ as $N\to\infty$. Letting $D_c(N) = \pi_c(p_N)/\pi(p_N)$ (the density of primes at most $p_N$ that are $c$-Ramanujan), to prove the theorem it suffices to show
\begin{equation}
|D_c(N) - (1 - c)|\ \ll \ \frac{\log\log N}{\log N},
\end{equation} which we now do.

From Theorem \ref{thm:asympbehRcn}(1), we know $R_{c,n}$ is asymptotic to $p_N$. Specifically, from \eqref{eq:closenessrcnpn} we find \be p_{\frac{n}{1-c}} - \beta_{3,c} n \log\log n \ \le \ R_{c,n} \ \le \ p_{\frac{n}{1-c}} + \beta_{3,c} n \log \log n.\ee As $n \approx (1-c)N$ with $c<1$, letting \be a_N  \ = \ p_N -\beta_{4,c} N \log \log N, \ \ \ b_N \ = \ p_N+\beta_{4,c} N \log\log N, \ee we find $R_{c,n} \in [a_N, b_N]$ for some $\beta_{4,c}$.



Note $D_c(N)$ is largest in the case where $R_{c,n} = a_N$ and every other prime up to $p_N$ is $c$-Ramanujan, and it is smallest if $R_{c,n} = b_N$ and no other prime in $[a_N, b_N]$ is $c$-Ramanujan. We show that the number of primes in $[a_N, b_N]$ is small relative to $\pi(p_N) = N$: \be\label{eq:needtoshowtofinishdnproof}\frac{\pi(b_N) - \pi(a_N)}{\pi(N)} \ \le \ \frac{\beta_{5,c} \pi(N) \frac{\log\log N}{\log N}}{\pi(N)} \ = \ \beta_{5,c} \frac{\log\log N}{\log N}; \ee as this tends to zero with $N$, the limiting probability $D_c(N)$ must exist and equal $1-c$.

We now prove \eqref{eq:needtoshowtofinishdnproof}. We trivially modify equations \eqref{eq:pipicx} and \eqref{eq:piminuspicx}, using $b_N$ and $a_N$ instead of $c$ and $cx$, and find, for some $q_N \in [a_N, b_N]$, that
\begin{eqnarray}
\pi (b_N) - \pi (a_N) & \ =\ & \text{Li}'(q_N)(b_N - a_N) + {O}\left(\frac{b_N}{\log^3 a_N} \right)\nonumber\\ & \ \le \ & \frac{2\beta_{4,c} N \log\log N}{\log a_N} + O\left(\frac{b_N}{\log^3 a_N}\right).
\end{eqnarray} Using Rosser's theorem (see \eqref{eq:pmvsmlogmfirst}), we find $b_N \le 2N\log N$ and $a_N \ge \frac12N \log N$ for $N$ large, implying that \be \pi(b_N) - \pi(a_N) \ \le \ \frac{\beta_{5,c} N \log \log N}{\log N}  \ee for some $\beta_{5,c}$. Dividing by $\pi(N) = N$ completes the proof of Theorem \ref{thm:asympbehRcn}.\hfill $\Box$




\section{Distribution of generalized Ramanujan primes}

\subsection{Numerical Simulations}

In this section we numerically explore how the $c$-Ramanujan primes are distributed among the primes, extending the work of Sondow, Nicholson and Noe \cite{SNN}.

In Table \ref{table:probcramanujan} we checked to see if numerical simulations for various $c$ and primes up to $10^6$ agree with our asymptotic behavior predictions.

\begin{table}

\begin{center}

\begin{tabular}[t]{|c|c|c|c|}

\hline

\multicolumn{4}{|c|}{Probability of a prime being $c$-Ramanujan}\\\hline

c & Expected density & Actual density & Ratio $R_{c,n}/ p_{\frac{n}{1-c}}$ \\\hline

0.05 & 0.95 & 0.9346 & 1.0181\\

0.10 & 0.90 & 0.8778 & 1.0280\\

0.15  & 0.85 & 0.8236 & 1.0353 \\

0.20  & 0.80 & 0.7709 & 1.0413 \\

0.25 & 0.75 & 0.7192 & 1.0470 \\

0.30 & 0.70 & 0.6688 & 1.0513 \\

0.35 & 0.65 & 0.6181 & 1.0567 \\

0.40 & 0.60 & 0.5687 & 1.0607 \\

0.45 & 0.55 & 0.5197 & 1.0641 \\

0.50 &  0.50 & 0.4708 & 1.0681 \\

0.55 &  0.45 & 0.4226  &   1.0712  \\

0.60 &  0.40 & 0.3745 & 1.0749  \\

0.65 & 0.35 & 0.3270 & 1.0774  \\

0.70  & 0.30 &  0.2797  &  1.0800 \\

0.75 & 0.25 & 0.2326 &  1.0821 \\

0.80 & 0.20 & 0.1853  & 1.0869 \\

0.85 & 0.15 & 0.1519 & 1.0897 \\

0.90 & 0.10 & 0.1013 & 1.0955 \\\hline

\end{tabular}
\caption{\label{table:probcramanujan} Expected density of $c$-Ramanujan primes amongst the prime numbers from Theorem \ref{density} and actual computed density or all $c$-Ramanujan primes less than $10^6$.  Ratio of largest $c$-Ramanujan prime in this interval to its asymptotic value from Theorem \ref{density}.}
\end{center}
\end{table}

We see the computations agree with our theoretical results. Note the ratio is closer to 1 for small values of $c$, which is plausible as we have more $c$-Ramanujan primes as data points in this same interval.

We also looked at runs of consecutive $c$-Ramanujan primes and non-$c$-Ramanujan primes in the sequence of primes; our results are summarized in Table \ref{table:consecutives}. The expected length of the maximum run was computed using a binomial coin flip model. Specifically, let $L_N$ be the random variable denoting the length of the longest sequence of consecutive heads obtained from tossing a coin with probability $P_c(N)$ of heads $N$ times, with the tosses independent. We have (see \cite{Sc} for the proof)
\begin{eqnarray}\label{eq:schimeanvar}
\mathbb{E}[L_N] & \ \approx\  & \frac{\log N}{\log (1/P_c(N))}- \left( \frac{1}{2}-\frac{\log (1-P_c(N))+\gamma}{\log (1/P_c(N))} \right) \nonumber\\ {\rm Var}(L_N) & = & \frac{\pi^2}{6 \log^2\left(1/P_c(N)\right)} + \frac{1}{12} + r_2(N) + o_c(N),
\end{eqnarray}
where $\gamma=0.5772 \ldots$ is the Euler-Mascheroni constant and $|r_2(N)| \le .00006$. Here $P_c(N)$ is the ratio of the number of $c$-Ramanujan primes to the total number of primes in the interval $(10^5, 10^6]$, and $N = \pi(10^6) - \pi(10^5)$ is the total number of primes in the interval.

Although we are assuming the probability of a prime being $c$-Ramanujan to be constant throughout the interval, the probability actually varies because the density of $c$-Ramanujans is greater in some intervals than others.   In Schilling's paper \cite{Sc}, the probability $P$ is constant as it represents the probability of getting a head when performing biased coin tosses.  In Table \ref{table:consecutives}, we take the interval $(10^5, 10^6]$ because the density will vary less than over the entire interval $[1, N)$. The actual probability of a prime being a $c$-Ramanujan prime is just the ratio of the number of $c$-Ramanujan primes in the interval $(10^5, 10^{6}]$ to the total number of primes in that interval.

We notice that for $c$ near $1/2,$ runs of non-$c$-Ramanujan primes are longer than predicted. Also striking is the large discrepancy in the length of the largest run for expected versus actual $c$-Ramanujan primes for small values of $c$ (and the related statement for $c$ near 1).

\begin{table}
\begin{center}
\begin{tabular}[t]{|c|c|c|c|c|}

\hline

& \multicolumn{4}{|c|}{Length of the longest run in $(10^5, 10^6)$ of}\\

 & \multicolumn{2}{|c|}{$c$-Ramanujan primes}& \multicolumn{2}{|c|}{Non-Ramanujan primes} \\

c & Expected & Actual & Expected & Actual \\\hline

0.05 & 127 & 97 & 4 & 2  \\

0.10 & 70 & 58 & 5 & 3\\

0.15  & 49 & 42 & 6 & 6 \\

0.20  & 38 & 36 & 7 & 7 \\

0.25 & 30 & 27 & 9 & 12\\

0.30 & 25 & 25 & 10 & 12\\

0.35 & 21 & 18 & 11 & 18\\

0.40 & 18 & 21 & 13 & 16\\

0.45 & 16 & 19 & 14 & 23 \\

0.50 &  14 & 20& 16 & 36 \\

0.55 & 12 & 16 & 19 & 39 \\

0.60 & 11 & 17 & 22 & 42 \\

0.65 & 10 & 13 & 25 & 53 \\

0.70 & 9 & 14 & 30 & 78 \\

0.75 & 8 & 11 & 37 & 119 \\

0.80 & 7 & 9 & 46 & 154 \\

0.85 & 6 & 10 & 62 & 303 \\

0.90 & 5 & 11 & 91 & 345 \\\hline

\end{tabular}

\caption{\label{table:consecutives} Length of the longest run of $c$-Ramanujan and non-Ramanujan primes in $(10^5, 10^6$)}

\end{center}

\end{table}

While the discrepancies for extreme values of $c$ are the largest, it is important to note that the variance in the coin flip model, though bounded independent of $N$ with respect to $N$ (see \eqref{eq:schimeanvar}), does vary significantly with respect to $c$.  Indeed, the closer $c$ is to 0 or 1, the larger is the probability of either being $c$-Ramanujan (for small $c$) or non-$c$-Ramanujan (for large $c$).  As such, the variance here can be on the order of $10^2$ or higher, explaining the very large deviations at the beginning and end of the table.  However, even accounting for this, the deviations are often twice the variance, which is an exceedingly large deviation.

Consider the case of $c = 0.8$.  If we look at the $c$-Ramanujan primes in the interval $[1,10^5]$ we see the density is 0.1852.  In the interval $[0,10^6]$, the density is 0.1830, and in the interval $[10^5, 10^6]$ the density is 0.1856.  As such, it is clear that the probability of being $c$-Ramanujan is almost constant in the interval $[10^5,10^6]$, and the expected longest run differs by at most 1 depending on which probability we use for a prime being a $c$-Ramanujan prime.

\subsection{Description of the algorithm}\label{alg}

To compute $c$-Ramanujan primes, we make slight modifications to the algorithm proposed in \cite{SNN} for generating 0.5-Ramanujan primes.  The algorithm is identical, with the exception of two minor details.  We first reprint the description of the algorithm from \cite{SNN}.

\begin{quote}
To compute a range of Ramanujan primes $R_i$ for $1 \leq i \leq n$, we perform simple calculations in each interval $(k/2, k]$ for $k = 1, 2, . . . , p_{3n - 1}$. To facilitate the calculation, we use a counter $s$ and a list $L$ with $n$ elements $L_i$. Initially, $s$ and all $L_i$ are set to zero. They are updated as each interval is processed. \\ \

After processing an interval, $s$ will be equal to the number of primes in that interval, and each $L_i$ will be equal either to the greatest index of the intervals so far processed that contain exactly $i$ primes, or to zero if no interval having exactly $i$ primes has yet been processed. \\ \

Having processed interval $k-1$, to find the number of primes in interval $k$ we perform two operations: add 1 to $s$ if $k$ is prime, and subtract 1 from $s$ if $k/2$ is prime. We then update the $s$-th element of the list to $L_s = k$, because now $k$ is the largest index of all intervals processed that contain exactly $s$ primes. \\ \

After all intervals have been processed, the list $R$ of Ramanujan primes is obtained by adding 1 to each element of the list $L$.
\end{quote}

We need to make two modifications to handle the case of general $c$. First, we need to adjust $s$ when incrementing $k$ corresponds to a change in $\pi(ck)$. In \cite{SNN}, the choice of $c = 0.5$ guarantees that the quantity $ck$ attains all the integers.  As such, to determine whether $\pi(ck)$ is incremented when $k$ is incremented, it sufficed to check whether the quantity $ck$ was prime or not.  Unfortunately, for many $c$ it is the case that not all integers are of the form $ck$ for some integer $k$. To correct for this, we check if the interval $(c(k-1) , ck]$ contains an integer.  If the interval does contain an integer, $m$, we check if $m$ is prime and adjust $s$ accordingly.

The second adjustment is with respect to the upper bound used for $R_{c,n}$.  We propose the following technique to obtain a crude upper bound dependent on $c$.

Using the following version of the prime number theorem (see \cite{RoSc})
\begin{eqnarray}
\frac{x}{\log x - \frac{1}{2}} \ < \  \pi (x) \ \ \ \text{for}\ 67 \leq x, \ \ \ \ \ \pi (x) \ <\  \frac{x}{\log x - \frac{3}{2}}  \ \ \ \text{for}\ e^{3/2} < x,
\end{eqnarray}
we have the following lower bound on the number of primes in the interval $(cx,x]$, for $x \geq \max\left(67, e^{3/2}/c\right)$:
\begin{eqnarray}
\pi (x) - \pi (cx)\ >\  \frac{x}{\log x - \frac{1}{2}} - \frac{cx}{\log x - A}\ =:\ f(x),
\end{eqnarray}
where we define the positive constant $A := -(\log c - \frac{3}{2})$.  It follows that an upper bound for $R_{c,n}$ can be obtained by finding an $x_0$ such that, for all $x \geq x_0$, we have $f(x)\ge n$.

To determine when this bound is monotonically increasing, we calculate the derivative to be
\begin{equation}
f'(x) \ = \ \frac{\log x - \frac{3}{2}}{(\log x - \frac{1}{2})^2} - c\frac{\log x - (A + 1)}{(\log x - A)^2}
\end{equation}
and determine for which values of $x$ is $f'(x)$ nonnegative.  Making the substitution $u = \log x - \frac{1}{2}$, we obtain the inequality
\begin{equation}
\left(u - 1\right) \left(u - \left(A - \frac{1}{2}\right) \right)^2 - c\left(u - \left(A + \frac{1}{2}\right)\right)u^2 \ \geq \ 0.
\end{equation}
This is a cubic inequality (with leading coefficient $1-c$ which is positive for all valid $c$), with trivially calculable roots, the greatest of which we denote $u_c$.  Then, for all $x > e^{u_c + \frac{1}{2}}$, the function $f(x)$ is monotonically increasing.

As such, the lower bound $f(x)$ is both valid and monotonically increasing for $x \geq \max \left( 67, e^{3/2}/c,  e^{u_c + \frac{1}{2}} \right) =: M_c$. Given a fixed $n_0$, we can solve numerically for $x_0$ by solving $f(x_0) = n_0$.  Provided that $x_0 > M_c$, we see that $x_0$ is a valid upper bound for $R_{c,n_0}$.  For large $c$, this crude upper bound is computationally inefficient, even for small $n$.  Furthermore, this upper bound is crude enough that for $c < 0.5$, it is often more efficient to use the more carefully derived upper bounds for $c = 0.5$ in \cite{So} (namely $p_{3n}$), since $R_{c_1,n} \leq R_{c_2,n}$ for $c_1 < c_2$.

These numerical calculations were performed in MATLAB.

\section{Open problems}

We end with some open problems. Laishram and Sondow \cite{La,So} proved the bounds $p_{2n}<R_n<p_{3n}$ for $n>1$. This result can be generalized to $R_{c,n}$. An interesting question is to find good choices of $a_c$ and $b_c$ such that $p_{a_cn} \le R_{c,n} \le p_{b_cn}$ for all $n$.  Of course, using variations on Rosser's Theorem (see \cite{RoSc}), we can (and do, particularly in Section \ref{alg}) derive bounds that work for large $n$, and then check by brute force whether these upper bounds hold for lower $n$.  However, this tells us nothing about the optimal choice $a_c$ and $b_c$ that hold for all $n$. Along these lines, another project would be to find the $c$- and $n$-dependence in the asymptotic relation $R_{c,n} \sim p_{\frac{n}{1-c}}$ well enough to predict the observed values in Table \ref{table:probcramanujan}.

For a given prime $p$, for what values of $c$ is $p$ a $c$-Ramanujan prime?
There are many ways to quantify this. One possibility would be to fix a denominator and look at all rational $c$ with that denominator.

Finally, is there any explanation for the unexpected distribution of $c$-Ramanujan primes amongst the primes in Table \ref{table:consecutives} (see also \cite{SNN}, Table 1)?  That is, for a given choice of $c$, is there some underlying reason that the length of the longest consecutive run of $c$-Ramanujan primes or the non-$c$-Ramanujan primes are distributed quite differently than expected? The predictions were derived using a coin-tossing model. This is similar to the Cramer model; while this does correctly predict many properties of the distribution of the prime numbers, it has been shown to give incorrect answers on certain scales (see for example \cite{MS}).







%
%

\begin{thebibliography}{99999999.}%
%
%

\bibitem[BS96]{BS}
E. Bach and J. Shallit, Algorithmic Number Theory, MIT Press, Cambridge, MA, 1996.

\bibitem[IK04]{IK}
H. Iwaniec and E. Kowalski, Analytic Number Theory, AMS
Colloquium Publications, Vol. 53, AMS, Providence, RI, 2004.

\bibitem[La10]{La} S. Laishram, \emph{On a conjecture on Ramanujan primes}, Int. J. Number Theory
\textbf{6} (2010), 1869--1873.

\bibitem[MS00]{MS} H. Montgomery and K. Soundararajan, \emph{Beyond Pair Correlation}, preprint
(2000) \texttt{http://arxiv.org/abs/math/0003234v1}.

\bibitem[Ra19]{Ra} S. Ramanujan, \emph{A proof of Bertrand's postulate}, J. Indian Math. Soc.
\textbf{11} (1919), 181--182.

\bibitem [Ro38]{Ro}
\newblock J. B. Rosser, \emph{The $n$th prime is greater than $n$ ln $n$}, Proc. London Math. Soc.
\textbf{45} (1938), 21--44.

\bibitem[RoSc62]{RoSc}
\newblock J. B. Rosser and L. Schoenfeld, \emph{Approximate formulas for some functions of prime
numbers}, Illinois J. Math. \textbf{6} (1962), 64--94.

\bibitem[Sc90]{Sc} M. F. Schilling, \emph{The longest run of heads}, College Math. J. \textbf{21}
(1990), 196--207.

\bibitem[Sh09]{Sh}
V. Shevelev, \emph{Ramanujan and Labos primes, their generalizations and classifications of primes},
preprint (2009), \texttt{http://arxiv.org/abs/0909.0715}.


\bibitem[OEIS]{OEIS} N. J. A. Sloane, The On-Line Encyclopedia of Integer Sequences, published
electronically at \texttt{https://oeis.org}.

\bibitem[So09]{So} J. Sondow, \emph{Ramanujan primes and Bertrand's postulate}, Amer. Math. Monthly
\textbf{116} (2009), 630--635.

\bibitem[SNN11]{SNN} J. Sondow, J. W. Nicholson, T. D. Noe, \emph{Ramanujan Primes: Bounds, Runs,
Twins, and Gaps},
J. Integer Seq. \textbf{14} (2011), Article 11.6.22; version with corrected
Table 1 available at \texttt{http://arxiv.org/abs/1105.2249}.


\ \\

\end{thebibliography}
%

\end{document}